\documentclass[11pt]{article}%
\usepackage[nottoc]{tocbibind}
\usepackage{xcolor}
\usepackage{amsfonts,amsmath,amssymb}
\usepackage{graphicx}
\usepackage{geometry}
\usepackage{color}
\usepackage{amsmath}
\usepackage{amsfonts}
\usepackage{amssymb}
\usepackage{caption}%
\providecommand{\U}[1]{\protect\rule{.1in}{.1in}}
\newtheorem{theorem}{Theorem}

\newtheorem{definition}[theorem]{Definition}

\newtheorem{lemma}[theorem]{Lemma}

\newtheorem{proposition}[theorem]{Proposition}
\newtheorem{remark}[theorem]{Remark}

\newenvironment{proof}[1][Proof]{\noindent\textbf{#1.} }{\ \rule{0.5em}{0.5em}}
\begin{document}

\title{Comet and moon solutions in the time-dependent \\restricted $(n+1)$-body problem}
\author{Carlos Barrera \thanks{Depto. Matem\'{a}ticas y Mec\'{a}nica IIMAS,
Universidad Nacional Aut\'{o}noma de M\'{e}xico, Apdo. Postal 20-726, 01000
Ciudad de M\'{e}xico, M\'{e}xico. crba@ciencias.unam.mx}, Abimael Bengochea
\thanks{Department of Mathematics, ITAM, R\'io Hondo 1, 01080, Ciudad de
M\'exico, M\'exico. abimael.bengochea@itam.mx} , Carlos Garc\'{\i}a-Azpeitia
\thanks{Depto. Matem\'{a}ticas y Mec\'{a}nica IIMAS, Universidad Nacional
Aut\'{o}noma de M\'{e}xico, Apdo. Postal 20-726, 01000 Ciudad de M\'{e}xico,
M\'{e}xico. cgazpe@mym.iimas.unam.mx} }
\maketitle

\begin{abstract}
The time-dependent restricted $(n+1)$-body problem concerns the study of a
massless body (satellite) under the influence of the gravitational field
generated by $n$ primary bodies following a periodic solution of the $n$-body
problem. We prove that the satellite has periodic solutions close to the
large-amplitude circular orbits of the Kepler problem (comet solutions), and
in the case that the primaries are in a relative equilibrium, close to
small-amplitude circular orbits near a primary body (moon solutions). The
comet and moon solutions are constructed with the application of a
Lyapunov-Schmidt reduction to the action functional. In addition, using
reversibility technics, we compute numerically the comet and moon solutions
for the case of four primaries following the super-eight choreography.

\end{abstract}

\section{Introduction}

The $n$-body problem consists of $n$ masses interacting under gravitational
forces. Due to the complexity of the $n$-body problem, particular cases of
this problem have been thoroughly studied such as the \textit{restricted
$(n+1)$-body problem}, which consists of $n$ primary bodies that follow a
general solution of the $n$-body problem and an extra body of negligible mass
(satellite). The satellite does not influence the movement of the $n$
primaries but it is influenced by the gravitational forces of the $n$ primaries.

A simplification of the restricted $(n+1)$-body problem assumes that the $n$
primaries are in a relative equilibrium (see \cite{ArEl04,BaEl04,GaIz10,Ka08,
LS11,MoZe}). This problem is a generalization of the classical
\emph{restricted }$3$\emph{-body problem} studied in \cite{DoRo07,
GoLlSi00,MaRy04, MeHa91,MoZe}. Other works have considered the case in which
the $n$ primaries follow a homographic elliptic solution
\cite{BeMaOv2017,Broucke1969, SzeGia64}. The case of the so-called
\emph{Sitnikov problem} considers the case of $2$ primaries. While the
analytic study of the restricted $(n+1)$-body problem for $n$ primaries
following elliptic homographic solutions has been the focus of many research
papers, few results have considered the problem where the primaries describe a
general periodic solution \cite{Lara2019}.

In this work, we present an analytical study of the \emph{time-dependent
restricted }$(n+1)$\emph{-body problem} for a general homogeneous potential,
where the primary bodies describe a general periodic solution in the plane.
Specifically, the Newton equation for a satellite with position $q(t)\in
\mathbb{R}^{2}$ is%
\begin{equation}
\ddot{q}(t)=-\sum_{j=1}^{n}m_{j}\frac{q(t)-q_{j}(t)}{\left\Vert q(t)-q_{j}%
(t)\right\Vert ^{\alpha+1}},
\end{equation}
where $q_{j}(t)\in\mathbb{R}^{2}$ represents the position of the $j$th body
with mass $m_{j}$ and $\Vert\cdot\Vert$ is the euclidean norm. We assume that
$\alpha\geq1$, where $\alpha=2$ is the gravitational case. We also consider,
without loss of generality, that the solution of the $n$-body problem
$q_{j}(t)$ is $2\pi$-periodic. Two limiting problems for the satellite will be
considered. In the first case, the satellite is far from the primary bodies
(\textit{the comet problem)}. In the second case, the satellite is close to
one of the primary bodies (\textit{the moon problem)}.

Specifically, in Theorem \ref{comet} we prove that for each integer
$\mathfrak{p}$ there is an integer $\mathfrak{q}_{0}$ such that for each
integer $\mathfrak{q}>\mathfrak{q}_{0}$, the comet has at least two
$2\pi\mathfrak{q}$-periodic solutions$\ $of the form
\[
q(t)=\varepsilon^{-1}e^{J\left(  \theta+\mathfrak{p}t/\mathfrak{q}\right)
}x_{0}+\mathcal{O}(\varepsilon),\qquad\varepsilon=\left(  \mathfrak{p}%
/\mathfrak{q}\right)  ^{2/(\alpha+1)},
\]
where $x_{0}=(1,0)\in\mathbb{R}^{2}$, $J$ is the symplectic matrix, $\theta
\in\lbrack0,2\pi]$ is a phase determined by the periodic solution of the
primaries and $\mathcal{O}(\varepsilon)$ is a $2\pi\mathfrak{q}$-periodic
function of order $\varepsilon$. The amplitude of the solution $\varepsilon
^{-1}$ is large and the frequency $\mathfrak{p}/\mathfrak{q}$ is small. Thus
the comet winds around the origin $\mathfrak{p}$ times while the primary
bodies travel their orbits $\mathfrak{q}$ times.

In the moon problem we require that the primaries are in a relative
equilibrium. Under this assumption we prove in Theorem \ref{moon} that for
each integer $\mathfrak{q}$ there is an integer $\mathfrak{p}_{0}$ such that
for each integer $\mathfrak{p}>\mathfrak{p}_{0}$, the moon has at least two
$2\pi\mathfrak{q}$-periodic solutions$\ $of the form
\[
q(t)=q_{1}(t)+\varepsilon e^{J(\theta+\mathfrak{p}t/\mathfrak{q})}%
x_{0}+\mathcal{O}(\varepsilon^{3}),\qquad\varepsilon=\left(  \mathfrak{p}%
/\mathfrak{q}\right)  ^{-2/(\alpha+1)}~.
\]
The amplitude of the solution $\varepsilon$ around the first primary body is
small and the frequency $\mathfrak{p}/\mathfrak{q}$ is large. Thus the moon
winds around one of the primaries $\mathfrak{p}$ times while the primaries
travel their periodic orbits $\mathfrak{q}$ times.

The regularized action functional of each of the problems is written (in
proper coordinates) as
\[
\mathcal{A}=\mathcal{A}_{0}+\mathcal{H}:\Omega\subset H_{2\pi}^{1}%
(\mathbb{R}^{2})\rightarrow\mathbb{R},
\]
where $H_{2\pi}^{1}$ is the Sobolev space of $2\pi$-periodic functions. Here,
$\mathcal{A}_{0}$ is the action functional of the Kepler problem and
$\mathcal{H}$ is an action functional of order $\varepsilon$. The functional
$\mathcal{A}_{0}$ has a $S^{1}$-set of critical points that consists of the
circular orbits of the Kepler problem. The set $\Omega$ is neighborhood of the
$S^{1}$-set in $H_{2\pi}^{1}$. We prove that $\mathcal{A}$ has at least two
critical points that persist from the isolated $S^{1}$-set of critical points
of the functional $\mathcal{A}_{0}$ for a small parameter $\varepsilon$. This
argument is based on a Lyapunov-Schmidt reduction similarly to \cite{FoGa},
where braids of the full $N$-body problem are constructed by replacing a body
in a \emph{central configuration} by two bodies (see also \cite{Fo}).

The gravitational case $\alpha=2$ is special because the circular orbits of
the Kepler problem are not isolated due to the existence of elliptic orbits.
The gravitational case is treated in Theorem \ref{grav} under the assumption
that the primaries form $m$-polygons at any time. This condition is satisfied
by many choreographies found in \cite{CaDoGa18b}, for example the super-eight
choreography with $n=4$ (see Remark \ref{ReC}), and also by many central
configurations of nested polygons with a central body (see Remark \ref{ReM}).

\begin{figure}[h]
\centering
\captionsetup{width=.9\linewidth}
\includegraphics[width=70mm]{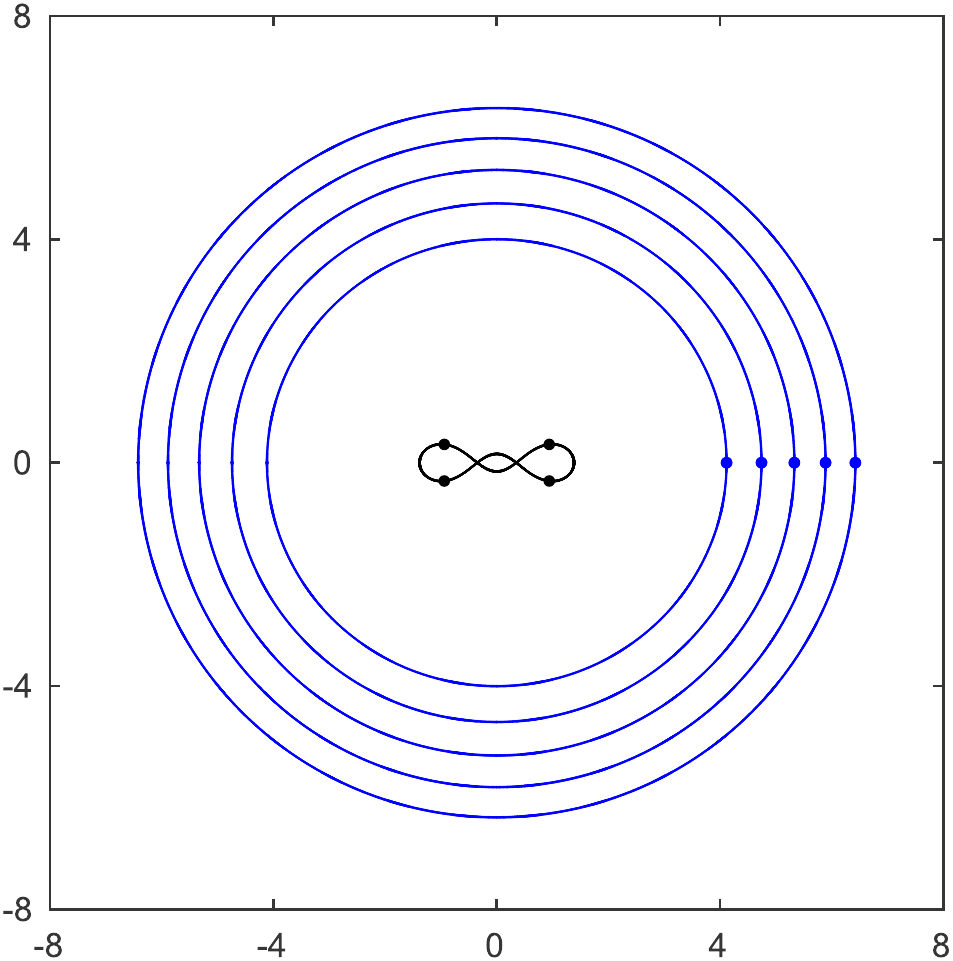}\includegraphics[width=70mm]{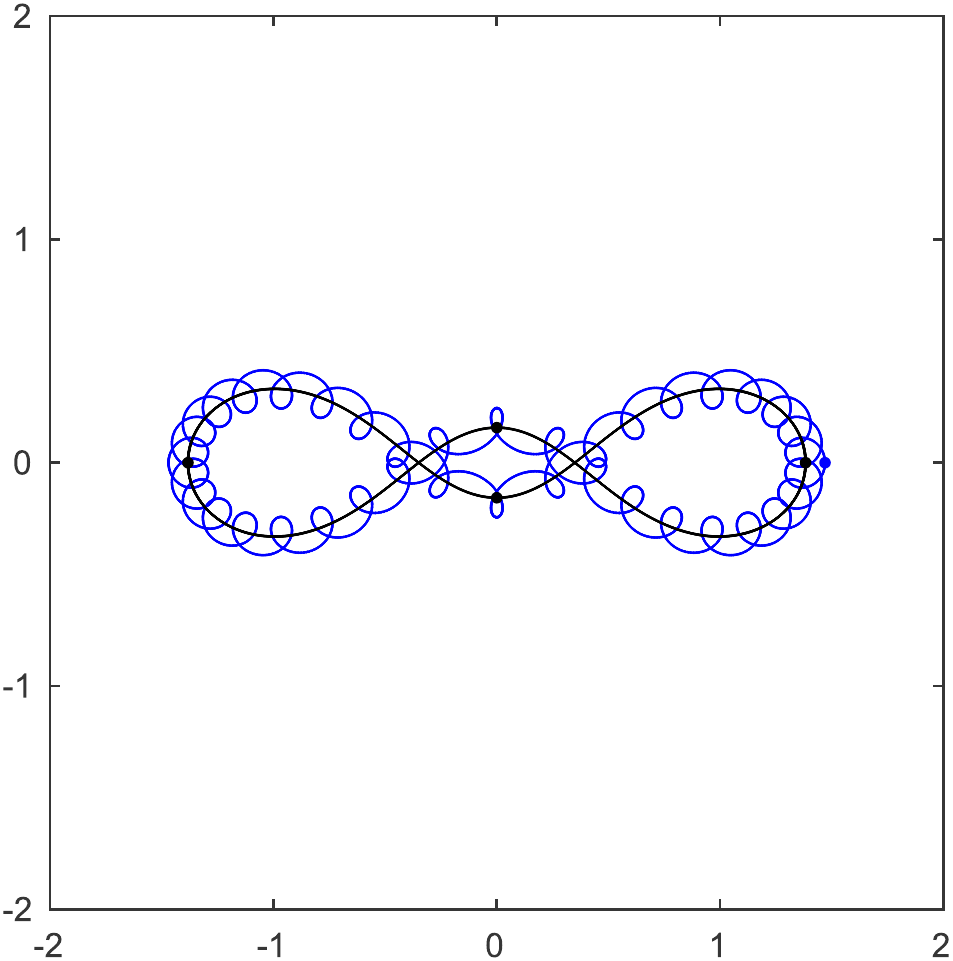}\caption{Left:
Multiple symmetric periodic orbits of comet type in the restricted 5-body
problem. Right: A symmetric periodic orbit of moon type in the restricted
5-body problem.}%
\label{figure1}%
\end{figure}

The method used to prove the existence of comet and moon solutions in the
time-dependent restricted body problem has limited applicability due to the
assumptions: (i) In the moon case, the periodic solution of the primaries
needs to be a relative equilibrium. (ii) In the gravitational case $\alpha=2$,
the periodic solution of the primaries needs to be symmetric. In Section 5, we
discuss briefly how to apply reversibility technics in order to show
numerically the existence of comet orbits (obtained in Theorem \ref{grav}) and
moon orbits that cannot be obtained with our theorems. Specifically, we
consider the case where four bodies with unitary mass follow the Gerver's
super-eight choreography \cite{Tomasz2003,Munoz2006,Mitsuru2014}. The fifth
body is the massless particle whose dynamics is defined by the force exerted
by the other four (Fig. \ref{figure1}). Technics of reversibility have been
successfully applied for studying periodic orbits of ordinary differential
equations \cite{Lamb1998}; see \cite{beng2013,beng2015,galan,Munoz1998} for
the case of the $N$-body problem. For more details on reversibility technics,
the interested reader is referred to \cite{Lara2019}, where comet and moon
orbits have been computed numerically for three primary bodies following the
eight choreography.

The rest of the paper proceeds as follows. In Section 2, we set the change of
variables and necessary hypothesis in order to write the functional
$\mathcal{A}$ as a perturbation of the functional for the Kepler problem
$\mathcal{A}_{0}$. In Section 3, we estimate the spectrum of the Hessian of
$\mathcal{A}_{0}$ in Fourier components (Proposition \ref{estimates}), in
order to make a Lyapunov-Schmidt reduction to a finite dimension. In Section
4, we prove the existence of comet and moon solutions as a consequence of our
main result (Theorem \ref{main}). Theorem \ref{main} cannot be applied
directly in the gravitational case; for this reason, we obtain a separate
result for the gravitational case under additional hypothesis. In Section 5,
we use the reversibility technics to compute numerically the comet and moon
solutions for four primaries following the super-eight choreography.

\section{Setting up the problem}

Let $q_{j}(t)\in\mathbb{R}^{2}$ be the positions of $n$ bodies with masses
$m_{j}$ for $j=1,...,n$. We assume that $q_{j}(t)$ is a periodic solution of
the $n$-body problem interacting under a general homogeneous potential. After
rescaling space and time, we can assume, without loss of generality, that the
solution $q_{j}(t)$ is $2\pi$-periodic and that the center of mass is at the
origin%
\[
\sum_{j=1}^{n}m_{j}q_{j}(t)=0,\qquad\sum_{j=1}^{n}m_{j}q_{j}(t)=0.
\]

The Newton equation for a satellite is%
\begin{equation}
\ddot{q}(t)=-\nabla_{q}U(q,t)=-\sum_{j=1}^{n}m_{j}\frac{q(t)-q_{j}%
(t)}{\left\Vert q(t)-q_{j}(t)\right\Vert ^{\alpha+1}}, \label{body}%
\end{equation}
where
\[
U(q,t)=-\sum_{j=1}^{n}m_{j}\phi_{\alpha}\left(  \left\Vert q-q_{j}\right\Vert
\right)  ,\qquad\phi_{\alpha}(\lambda)=\frac{1}{\alpha-1}\lambda^{1-\alpha}%
\]
when $\alpha>1$ and $\phi_{1}(\lambda)=-\log(\lambda).$

We define $\mathcal{A}_{0}$ as the action for the Kepler problem in rotating
coordinates,
\begin{equation}
\mathcal{A}_{0}(x)=\int_{0}^{2\pi}\left(  \frac{1}{2}\left\Vert \left(
\frac{\nu}{\omega}\partial_{\tau}+J\right)  x(\tau)\right\Vert ^{2}%
+\phi_{\alpha}(\left\Vert x(\tau)\right\Vert )\right)  ~d\tau,\qquad J=%
\begin{pmatrix}
0 & 1\\
-1 & 0
\end{pmatrix}
\text{.} \label{kepler}%
\end{equation}
The objective of this section is to write the action for the equation
\eqref{body} as $\mathcal{A}=\mathcal{A}_{0}+\mathcal{H}$, where $\mathcal{H}$
is a small perturbation with $\mathcal{H}(x)=\mathcal{O}(\varepsilon)$. The
parameter $\varepsilon^{-1}$ represents the amplitude of the comet and
$\varepsilon$ the distance of the moon to a primary body.

\subsection{The comet problem}

We assume in this section, without loss of generality, that%
\[
M=\sum_{j=1}^{n}m_{j}=1.
\]
The Kepler problem
\[
\ddot{q}=-\frac{q}{\left\Vert q\right\Vert ^{\alpha+1}}%
\]
admits solutions of the form
\[
q(t)=\varepsilon^{-1}e^{J\omega t}%
\begin{pmatrix}
1\\
0
\end{pmatrix}
,
\]
where the frequency $\omega$ and amplitude $\varepsilon^{-1}$ satisfy the
relation $\omega^{2}=\varepsilon^{(\alpha+1)}$. Note that $\varepsilon^{-1}$
is proportional to the distance between the satellite $q$ and the origin. To
obtain a continuation of the large-amplitude solutions ($\varepsilon
\rightarrow0$), we make the change of variables that scales the Kepler problem
as $t=\tau/\nu$ and
\begin{align}
q(t)  &  =\varepsilon^{-1}e^{J\omega\tau/\nu}x(\tau),\label{cambio_q_c}\\
q_{j}(t)  &  =e^{J\omega\tau/\nu}x_{j}(\tau),\enskip j=1,\cdots,n,
\label{cambio_qj_c}%
\end{align}
with $\omega^{2}=\varepsilon^{\alpha+1}$.

\begin{proposition}
Set $\omega^{2}=\varepsilon^{(\alpha+1)}$. For $\alpha\geq1$, the solutions of
restricted problem \eqref{body} in the coordinate $x(\tau)$ (given by
\eqref{cambio_q_c}) are critical points of the action
\[
\mathcal{A}(x)=\mathcal{A}_{0}(x)+\mathcal{H}(x),
\]
where $\mathcal{H}(x)=\int_{0}^{2\pi}h(x(\tau),\tau)d\tau$ with%
\[
h(x,\tau)=\sum_{j=1}^{n}m_{j}\left[  \phi_{\alpha}(\left\Vert x(\tau
)-\varepsilon x_{j}(\tau)\right\Vert )-\phi_{\alpha}(\left\Vert x(\tau
)\right\Vert )\right]  \text{.}%
\]
and $x_{j}(\tau)$ are given by \eqref{cambio_qj_c}.
\end{proposition}

\begin{proof}
Using the change of variables \eqref{cambio_q_c} and $t=\tau/\nu$, the first
term of \eqref{body} becomes
\[
\ddot{q}(t)=\varepsilon^{-1}\omega^{2}e^{J\omega\tau/\nu}\left(  \dfrac{\nu
}{\omega}\partial_{\tau}+J\right)  ^{2}x(\tau).
\]
Replacing this term in \eqref{body} and using \eqref{cambio_qj_c}, the
equation of motion becomes%
\begin{equation}
\left(  \frac{\nu}{\omega}\partial_{\tau}+J\right)  ^{2}x(\tau)=-\sum
_{j=1}^{n}m_{j}\frac{x(\tau)-\varepsilon x_{j}(\tau)}{\left\Vert
x(\tau)-\varepsilon x_{j}(\tau)\right\Vert ^{\alpha+1}}. \label{eqcomet}%
\end{equation}
On the other hand, if $\mathcal{A}(x)=\int_{0}^{2\pi}A(\tau,x(\tau),x^{\prime
}(\tau))d\tau$, then
\begin{align*}
\dfrac{d}{d\tau}\dfrac{\partial A}{\partial x^{\prime}}  &  =\dfrac{\nu^{2}%
}{\omega^{2}}x^{\prime\prime}(\tau)+\dfrac{\nu}{\omega}Jx^{\prime}(\tau),\\
\dfrac{\partial A}{\partial x}  &  =-\dfrac{\nu}{\omega}Jx^{\prime}%
(\tau)+x(\tau)-\sum_{j=1}^{n}m_{j}\dfrac{x(\tau)-\varepsilon x_{j}(\tau
)}{\left\Vert x(\tau)-\varepsilon x_{j}(\tau)\right\Vert ^{\alpha+1}}.
\end{align*}
Using the Euler-Lagrange equation, we get equation \eqref{eqcomet}.
\end{proof}

We need to define $h(x,\tau)$ in the space of $2\pi$-periodic functions.

\begin{proposition}
If
\[
\omega=\mathfrak{p}/\mathfrak{q},\qquad\nu=1/\mathfrak{q},
\]
then the function $h(t,x)$ defined in Proposition 1 is $2\pi$-periodic and
$\nabla_{x}h(x,\tau)=\mathcal{O}(\varepsilon^{2})$.
\end{proposition}

\begin{proof}
Since $q_{j}(t)$ is $2\pi$-periodic, then $q_{j}(\tau/\nu)$ is $2\pi
/\mathfrak{q}$-periodic and $e^{-J\omega t/\nu}$ is a $2\pi/\mathfrak{p}$-periodic matrix.
Therefore, $x_{j}(\tau)$ is $2\pi$-periodic and $h(x,\tau)=h(x,\tau+2\pi)$.
Since $\sum_{j=1}^{n}m_{j}x_{j}=0$, using complex identification of
$x\in\mathbb{C}$, then%
\begin{align*}
\nabla_{x}h(x,\tau)  &  =\sum_{j=1}^{n}m_{j}\left(  \alpha\frac{x}{\left\vert
x\right\vert ^{\alpha+2}}x_{j}\varepsilon+\frac{\alpha(\alpha+1)}{2}\frac
{x}{\left\vert x\right\vert ^{\alpha+3}}x_{j}^{2}\varepsilon^{2}%
+\mathcal{O}(\varepsilon^{3})\right) \\
&  =\left(  \sum_{j=1}^{n}m_{j}\frac{\alpha(\alpha+1)}{2}\frac{x}{\left\vert
x\right\vert ^{\alpha+3}}x_{j}^{2}\right)  \varepsilon^{2}+\mathcal{O}%
(\varepsilon^{3})\text{.}%
\end{align*}

\end{proof}

We look for comet solutions in which the amplitude is very large. That is, we
need to take $\varepsilon\rightarrow0$ and $\omega^{2}=\varepsilon
^{(\alpha+1)}\rightarrow0$. Because the equation of motion for $x(\tau)$ has
the term $\left(  \frac{\nu}{\omega}\partial_{\tau}+J\right)  ^{2}x(\tau)$ and
$\omega\rightarrow0$, we need $\nu/\omega\geq\delta>0.$ Since $\nu/\omega
=1/\mathfrak{p}$, we achieve this by fixing $\mathfrak{p}$ and letting $\mathfrak{q}\rightarrow\infty$.

\subsection{The moon problem}

We assume that the satellite follows the first body. After rescaling space and
time we can assume, without loss of generality, that $m_{1}=1$. To find the
moon solutions, we define the time scale $t=\tau/\nu$ and the change of
variables
\begin{align}
q(t)  &  =q_{1}(t)+\varepsilon e^{J\omega\tau/\nu}x(\tau),\label{cambio_q_m}\\
q_{j}(t)  &  =e^{J\omega\tau/\nu}x_{j}(\tau),\enskip j=1,\cdots,n,
\label{cambio_qj_m}%
\end{align}
where $\omega^{2}=\varepsilon^{-(\alpha+1)}$.

\begin{proposition}
Let $\alpha\geq1$ and $\omega^{2}=\varepsilon^{-(\alpha+1)}$. The solutions of
restricted problem \eqref{body} in the coordinate $x(\tau)$ (given by
\eqref{cambio_q_m}) are critical points of the action%
\[
\mathcal{A}(x)=\mathcal{A}_{0}(x)+\mathcal{H}(x),
\]
where $\mathcal{H}(x)=\int_{0}^{2\pi}h(x,\tau)d\tau$ with
\[
h(x,\tau)=\left\{
\begin{matrix}
\varepsilon^{\alpha-1}\displaystyle\sum_{j=2}^{n}m_{j}\left[  \phi_{\alpha
}\left(  \left\Vert x_{1}(\tau)-x_{j}(\tau)+\varepsilon x(\tau)\right\Vert
\right)  -\nabla\phi_{\alpha}(\left\Vert x_{1}(\tau)-x_{j}(\tau)\right\Vert
)\cdot\varepsilon x(\tau)\right]  ,~\alpha>1\\
\displaystyle\sum_{j=2}^{n}m_{j}\left[  -\log\left(  \dfrac{\left\Vert
x_{1}(\tau)-x_{j}(\tau)+\varepsilon x(\tau)\right\Vert }{\left\Vert x_{1}%
(\tau)-x_{j}(\tau)\right\Vert }\right)  +\dfrac{x_{1}(\tau)-x_{j}(\tau
)}{\left\Vert x_{1}(\tau)-x_{j}(\tau)\right\Vert ^{2}}\cdot\varepsilon
x(\tau)\right]  ,~\alpha=1
\end{matrix}
\right.  ,
\]
and $x_{j}(\tau)$ are given by \eqref{cambio_qj_m}.
\end{proposition}

\begin{proof}
Using the change of variables \eqref{cambio_q_m}, the left side of the
equation \eqref{body} becomes
\[
\ddot{q}(t)=\ddot{q_{1}}(t)+\varepsilon\omega^{2}e^{J\omega\tau/\nu}\left(
\dfrac{\nu}{\omega}\partial_{\tau}+J\right)  ^{2}x(\tau).
\]
But $q_{1}(t)$ is a solution of the $n$-body problem, therefore
\[
\ddot{q_{1}}(t)=-\sum_{j=2}^{n}m_{j}\frac{q_{1}(t)-q_{j}(t)}{\left\Vert
q_{1}(t)-q_{j}(t)\right\Vert ^{\alpha+1}}\text{.}%
\]
Assuming $\alpha>1$ and expanding the sum in the right side of \eqref{body},%
\[
\sum_{j=1}^{n}m_{j}\frac{q(t)-q_{j}(t)}{\left\Vert q(t)-q_{j}(t)\right\Vert
^{\alpha+1}}=\dfrac{e^{J\omega\tau/\nu}x(\tau)}{\varepsilon^{\alpha}\Vert
x(\tau)\Vert^{\alpha+1}}+\sum_{j=2}^{n}m_{j}\frac{q(t)-q_{j}(t)}{\left\Vert
q(t)-q_{j}(t)\right\Vert ^{\alpha+1}}.
\]
Replacing this in \eqref{body} and using \eqref{cambio_qj_m}, the equation of
motion is
\begin{multline}
\left(  \frac{\nu}{\omega}\partial_{\tau}+J\right)  ^{2}x(\tau)=-\frac
{x(\tau)}{\left\Vert x(\tau)\right\Vert ^{\alpha+1}}\label{eqmoon}\\
-\varepsilon^{\alpha}\sum_{j=2}^{n}m_{j}\left(  \frac{x_{1}(\tau)-x_{j}%
(\tau)+\varepsilon x(\tau)}{\left\Vert x_{1}(\tau)-x_{j}(\tau)+\varepsilon
x(\tau)\right\Vert ^{\alpha+1}}-\frac{x_{1}(\tau)-x_{j}(\tau)}{\left\Vert
x_{1}(\tau)-x_{j}(\tau)\right\Vert ^{\alpha+1}}\right)  .
\end{multline}

On the other hand, if $\mathcal{A}(x)=\int_{0}^{2\pi}A(\tau,x(\tau),x^{\prime
}(\tau))d\tau$, then
\begin{align*}
\dfrac{d}{d\tau}\dfrac{\partial A}{\partial x^{\prime}}=  &  \dfrac{\nu^{2}%
}{\omega^{2}}x^{\prime\prime}(\tau)+\dfrac{\nu}{\omega}Jx^{\prime}(\tau),\\
\dfrac{\partial A}{\partial x}=  &  -\dfrac{\nu}{\omega}Jx^{\prime}%
(\tau)+x(\tau)+\dfrac{x(\tau)}{\left\Vert x(\tau)\right\Vert ^{\alpha+1}}\\
&  -\varepsilon^{\alpha}\sum_{j=1}^{n}m_{j}\dfrac{x_{1}(\tau)-x_{j}%
(\tau)-\varepsilon x(\tau)}{\left\Vert x_{1}(\tau)-x_{j}(\tau)-\varepsilon
x(\tau)\right\Vert ^{\alpha+1}}-\dfrac{x_{1}(\tau)-x_{j}(\tau)}{\left\Vert
x(\tau)-\varepsilon x_{j}(\tau)\right\Vert ^{\alpha+1}}.
\end{align*}
Using the Euler-Lagrange equation, we get equation \eqref{eqmoon}. The case
$\alpha=1$ is similar.
\end{proof}

We look for moon solutions in which the amplitude $\varepsilon$ is very small.
At this point, we notice that the problem of the moon is different from the
problem of the comet because $\omega\rightarrow\infty$ if $\varepsilon
\rightarrow0$. But we need $\nu/\omega>\delta>0$, then $\nu\rightarrow\infty$
as $\varepsilon\rightarrow0$. But on the other hand, the period of $x_{j}%
(\tau)$ goes to infinity and $h(x,\tau)$ is not $2\pi$-periodic. Therefore, we
cannot define $\mathcal{A}$ in the space of $2\pi$-periodic paths. In order to
avoid this problem, we can consider the simpler case where the periodic
solution of the primaries is a relative equilibrium. That means that
$q_{j}(t)=e^{tJ}a_{j}$, where
\begin{equation}
a_{j}=\sum_{k=1(k\neq j)}^{n}m_{k}\dfrac{(a_{j}-a_{k})}{\Vert a_{j}-a_{k}%
\Vert^{\alpha+1}},\qquad m_{1}=1,\qquad a_{j}\in\mathbb{R}^{2}\text{.}
\label{centralconf}%
\end{equation}
In this case, we can choose values of $\omega$ and $\nu$ such that $h(x,\tau)$
is $2\pi$-periodic.

\begin{proposition}
If
\[
\dfrac{\omega-1}{\nu}=\mathfrak{p}\in\mathbb{Z}\text{,}%
\]
and $q_{j}(t)=e^{tJ}a_{j}$ with \eqref{centralconf}, then $h(x,\tau)$ defined
in the Proposition 3 is $2\pi$-periodic. Moreover, we have and $\nabla
_{x}h(x,\tau)=\mathcal{O}(\varepsilon^{\alpha+1})$ and $\nu/\omega=\left(
1-\varepsilon^{\frac{\alpha+1}{2}}\right)  /\mathfrak{p}.$
\end{proposition}

\begin{proof}
If $q_{j}(t)=e^{tJ}a_{j}$ and the configuration $a_{j}$ satisfies
\eqref{centralconf}, then
\[
x_{j}(\tau)=e^{-J(\omega-1)\tau/\nu}a_{j}.
\]
Therefore $x_{j}(\tau)$ and $h(x,\tau)$ are $2\pi$-periodic. Furthermore, from
the definition of $h$ we have $\nabla h(x,\tau)=\mathcal{O}(\varepsilon
^{\alpha+1}).$
\end{proof}

\section{Lyapunov-Schmidt reduction}

The Euler-Lagrange equations of $\mathcal{A}_{0}$ is
\[
\frac{\delta\mathcal{A}_{0}}{\delta x}(x;\varepsilon)=\left(  \frac{\nu
}{\omega}\partial_{\tau}+J\right)  ^{2}x+\frac{x}{\left\Vert x\right\Vert
^{\alpha+1}}.
\]
This equation has the circle of solutions $e^{J\theta}x_{0}$ for $\theta
_{0}\in\lbrack0,2\pi]$, where $x_{0}=(1,0)^{T}$,
\[
S^{1}=\left\{  e^{J\theta}x_{0}\in H_{2\pi}^{1}\left(  \mathbb{R}^{2}\right)
:\theta\in\lbrack0,2\pi]\right\}  .
\]
The question that we would like to answer is if the critical solutions of
$\mathcal{A}(x;0)=\mathcal{A}_{0}(x;0)$ persist for $\varepsilon\neq0$. For
this purpose we define a neighborhood $\Omega_{\rho}$ of $S^{1}$ where the
functional $\mathcal{A}$ is well defined,
\begin{equation}
\Omega_{\rho}:=\left\{  x\in H_{2\pi}^{1}\left(  \mathbb{R}^{2}\right)  :\Vert
x-e^{J\theta}x_{0}\Vert_{H^{1}}<\rho,\qquad\theta\in\lbrack0,2\pi]\right\}
\label{domain}%
\end{equation}
Thus our two problems are set as critical points of the action functional
$\mathcal{A}_{0}$ perturbed by $\mathcal{H}$,%
\begin{equation}
\mathcal{A}(x;\varepsilon)=\mathcal{A}_{0}(x;\varepsilon)+\mathcal{H}%
(x;\varepsilon):\Omega_{\rho}\times\mathbb{R}\rightarrow\mathbb{R}.
\label{A=A_0+H in coordinates (u_0,u)}%
\end{equation}

\begin{proposition}
Assume that $\alpha\geq1$, and $\omega$ and $\nu$ satisfy the conditions of
Propositions 2 and 4. Thus the action functional
\eqref{A=A_0+H in coordinates (u_0,u)} that gives solutions to the restricted
body problem is well-defined on $\Omega_{\rho}$ for $\rho$ small enough.
\end{proposition}

\begin{proof}
It is necessary only to see that $\mathcal{A}$ is well-defined in
$\Omega_{\rho}$ and that critical solutions in $\Omega_{\rho}$ are critical
solutions in $H_{2\pi}^{1}$. Since we look for critical points of the action
$\mathcal{A}(x;\varepsilon)$ in the space $H_{2\pi}^{1}$, then $\mathcal{A}$
is well-defined in $H_{2\pi}^{1}$ only if $\omega$ and $\nu$ satisfies the
conditions of Propositions 2 and 4.

Since $H_{2\pi}^{1}\subset C_{2\pi}^{0}$, there is a constant $C$ such that
$\Vert\cdot\Vert_{C^{0}}\leq C\Vert\cdot\Vert_{H^{1}}$. If $x\in\Omega_{\rho}%
$, then
\[
\max_{t}|x(t)-e^{J\theta}x_{0}|=\Vert x-e^{J\theta}x_{0}\Vert_{C^{0}}\leq
C\Vert x-e^{J\theta}x_{0}\Vert_{H^{1}}<C\rho\text{.}%
\]
Therefore, we choose $\rho$ small enough such that the path of $x$ belongs to
an annulus with center at the origin. The nonlinear integrals of
$\mathcal{A}(x)$ are bounded in the region where $h(x,t)$ is analytic, which
holds in the neighborhood $\Omega_{\rho}$ because it excludes collisions by construction.
\end{proof}

\subsection{Non-degeneracy condition}

In the moon and comet problems the functions $\nu=\nu(\varepsilon)$ and
$\omega=\omega(\varepsilon)$ depend on $\varepsilon$ and satisfy
\[
\lim_{\varepsilon\rightarrow0}\dfrac{\nu}{\omega}=\,{\dfrac{1}{\mathfrak{p}}}.
\]
Since $\mathcal{H}=\mathcal{O}(\varepsilon)$, the action at $\varepsilon=0$
is
\[
\mathcal{A}(x;0)=\mathcal{A}_{0}(x)=\int_{0}^{2\pi} \frac{1}{2}\left\Vert
\left(  \frac{1}{\mathfrak{p}}\partial_{\tau}+J\right)  x(\tau)\right\Vert ^{2}%
+\phi_{\alpha}\left(  \left\Vert x(\tau)\right\Vert \right)  ~d\tau\text{.}%
\]
In this section we make necessary estimates on the Hessian $\nabla
^{2}\mathcal{A}_{0}(x;0)$ in order to perform a Lyapunov--Schmidt reduction.

For $x\in H_{2\pi}^{1}$, the Fourier components of $x$ are
\[
x(t)=\sum_{l\in\mathbb{Z}}\hat{x}_{l}e^{il\tau}\text{,\qquad}\hat{x}%
_{l}=\overline{\hat{x}}_{-l}\in\mathbb{R}^{2}\text{.}%
\]
We define the components of $x$ as%
\[
\xi=\hat{x}_{0}=\int_{0}^{2\pi}x(\tau)~d\tau\in X_{0},\qquad\eta=\sum_{l\neq
0}\hat{x}_{l}e^{il\tau}=x(\tau)-\int_{0}^{2\pi}x(\tau)~d\tau\in X\text{.}%
\]
If we decompose $\xi$ in polar coordinates, $\xi=r\left(  \cos\theta
,\sin\theta\right)  $, we can write $x\in H_{2\pi}^{1}$ as a function of the
variables $(\theta,r,\eta)$, given by
\[
x(\theta,r,\eta)=r\binom{\cos\theta}{\sin\theta}+\eta,\qquad(\theta,r,\eta
)\in\lbrack0,2\pi]\times\mathbb{R}^{+}\times X.
\]
In the new coordinates $(\theta,r,\eta)$ the action is $\mathcal{A}_{0}%
(\theta,r,\eta)=\mathcal{A}(\theta,r,\eta;0)$. In these coordinates the
elements in the critical circle $S^{1}$ of $\mathcal{A}_{0}$ are
\[
x(\theta,1,0)=e^{J\theta}x_{0}\text{.}%
\]

\begin{lemma}
For $\alpha\geq1$ the action functional for the Kepler problem has the
expansion
\begin{equation}
\mathcal{A}_{0}(\theta,r+1,\eta)=\frac{1}{2}\displaystyle\int_{0}^{2\pi
}(\alpha+1)r^{2}+\left\Vert \left(  \frac{1}{\mathfrak{p}}\partial_{\tau}+J\right)
\eta\right\Vert ^{2}-\Vert\eta\Vert^{2}+(\alpha+1)\left(
\begin{pmatrix}
\cos\theta\\
\sin\theta
\end{pmatrix}
\cdot\eta\right)  ^{2}d\tau+\mathcal{O}(\left\vert (\theta,r,\eta)\right\vert
^{3}). \label{reduction}%
\end{equation}

\end{lemma}

\begin{proof}
For $\alpha\neq1$ the functional action for the Kepler problem is
\[
\mathcal{A}_{0}(x)=\displaystyle\int_{0}^{2\pi}\frac{1}{2}\left\Vert \left(
\frac{1}{\mathfrak{p}}\partial_{\tau}+J\right)  x(\tau)\right\Vert ^{2}+\frac{1}%
{\alpha-1}\Vert x(\tau)\Vert^{1-\alpha}~d\tau.
\]
In order to compute the Hessian of the functional $\mathcal{A}$ around the
critical solution, in the coordinates $(\theta,r,\eta)$, we expand in Taylor
series the functional
\[
\mathcal{A}_{0}(\theta,r+1,\eta)=\displaystyle\int_{0}^{2\pi}\frac{1}%
{2}(r+1)^{2}+\frac{1}{2}\left\Vert \left(  \frac{1}{\mathfrak{p}}\partial_{\tau
}+J\right)  \eta\right\Vert ^{2}+\frac{1}{\alpha-1}\left(  \left\Vert (r+1)%
\begin{pmatrix}
\cos\theta\\
\sin\theta
\end{pmatrix}
+\eta\right\Vert ^{2}\right)  ^{\frac{1-\alpha}{2}}d\tau.
\]
We notice that
\[
\left(  \left\Vert (r+1)%
\begin{pmatrix}
\cos\theta\\
\sin\theta
\end{pmatrix}
+\eta\right\Vert ^{2}\right)  ^{\frac{1-\alpha}{2}}=\left(  1+2r+r^{2}+2(r+1)%
\begin{pmatrix}
\cos\theta\\
\sin\theta
\end{pmatrix}
\cdot\eta+\Vert\eta\Vert^{2}\right)  ^{\frac{1-\alpha}{2}}\text{.}%
\]
Using that
\[
(1+\rho)^{\frac{1-\alpha}{2}}=1-\frac{\alpha-1}{2}\rho+\frac{\alpha^{2}-1}%
{8}\rho^{2}+\mathcal{O}(\rho^{3}),
\]
we get the expansion%
\begin{align*}
\left(  \left\Vert (r+1)%
\begin{pmatrix}
\cos\theta\\
\sin\theta
\end{pmatrix}
+\eta\right\Vert ^{2}\right)  ^{\frac{1-\alpha}{2}}  &  =1-\frac{\alpha-1}%
{2}\left[  2r+r^{2}+2(r+1)%
\begin{pmatrix}
\cos\theta\\
\sin\theta
\end{pmatrix}
\cdot\eta+\Vert\eta\Vert^{2}\right] \\
&  ~~~+\frac{\alpha^{2}-1}{8}\left[  4r^{2}+4\left(
\begin{pmatrix}
\cos\theta\\
\sin\theta
\end{pmatrix}
\cdot\eta\right)  ^{2}\right]  +\mathcal{O}(\left\vert (\theta,r,\eta
)\right\vert ^{3}).
\end{align*}
Since $\int_{0}^{2\pi}\eta d\tau=0$, we obtain \eqref{reduction}.

For $\alpha=1$ the functional action is
\[
\mathcal{A}_{0}(x)=\displaystyle\int_{0}^{2\pi}\frac{1}{2}\left\Vert \left(
\frac{1}{\mathfrak{p}}\partial_{\tau}+J\right)  x(\tau)\right\Vert ^{2}-\log\left(  \Vert
x(\tau)\Vert\right)  ~d\tau.
\]
Therefore
\begin{align*}
\mathcal{A}_{0}(\theta,r+1,\eta)=\frac{1}{2}  &  \displaystyle\int_{0}^{2\pi
}(r+1)^{2}+\left\Vert \left(  \frac{1}{\mathfrak{p}}\partial_{\tau}+J\right)
\eta\right\Vert ^{2}\\
&  -\log\left(  1+2r+r^{2}+2(r+1)%
\begin{pmatrix}
\cos\theta\\
\sin\theta
\end{pmatrix}
\cdot\eta+\Vert\eta\Vert^{2}\right)  d\tau.
\end{align*}
Using that
\[
\log(1+\rho)=\rho-\frac{1}{2}\rho^{2}+\mathcal{O}(\rho^{3}),
\]
we get the expansion
\begin{align*}
\log\left(  1+2r+r^{2}+2(r+1)%
\begin{pmatrix}
\cos\theta\\
\sin\theta
\end{pmatrix}
\cdot\eta+\Vert\eta\Vert^{2}\right)   &  =r+\frac{1}{2}r^{2}+(r+1)%
\begin{pmatrix}
\cos\theta\\
\sin\theta
\end{pmatrix}
\cdot\eta+\frac{1}{2}\Vert\eta\Vert^{2}\\
&  ~~~-r^{2}-\left(
\begin{pmatrix}
\cos\theta\\
\sin\theta
\end{pmatrix}
\cdot\eta\right)  ^{2}+\mathcal{O}(\left\vert (\theta,r,\eta)\right\vert
^{3}).
\end{align*}
Using that $\int_{0}^{2\pi}\eta d\tau=0$, we obtain \eqref{reduction} with
$\alpha=1$.
\end{proof}

\begin{proposition}
\label{estimates} Assume that $\alpha\geq1$ ($\alpha\neq2$), then the Hessian
of the action functional $\mathcal{A}_{0}(\theta,r,\eta)$ at $(\theta,1,0)$
is
\[
\nabla^{2}\mathcal{A}_{0}(\theta,1,0)=%
\begin{pmatrix}
\nabla_{\theta}^{2}\mathcal{A}_{0} & 0\\
0 & \nabla_{(r,\eta)}^{2}\mathcal{A}_{0}%
\end{pmatrix}
,
\]
where
\[
\nabla_{\theta}^{2}\mathcal{A}_{0}(\theta,1,0)=0,\qquad\left\Vert \left(
\nabla_{(r,\eta)}^{2}\mathcal{A}_{0}(\theta,1,0)\right)  ^{-1}\right\Vert
_{H_{2\pi}^{1}\rightarrow H_{2\pi}^{1}}\leq C,
\]
and $C$ is a positive constant.
\end{proposition}

\begin{proof}
By a direct computation from \eqref{reduction}, the Hessian is block diagonal
\[
\nabla^{2}\mathcal{A}_{0}(\theta,1,0)=\nabla_{\theta}^{2}\mathcal{A}_{0}%
\oplus\nabla_{r}^{2}\mathcal{A}_{0}\oplus\nabla_{\eta}^{2}\mathcal{A}_{0},
\]
where $\nabla_{\theta}^{2}\mathcal{A}_{0}=0$, $\nabla_{r}^{2}\mathcal{A}%
_{0}=\alpha+1$ and
\[
\nabla_{\eta}^{2}\mathcal{A}_{0}=(-\partial_{\tau}^{2}+1)^{-1}\left[  -\left(
\frac{1}{\mathfrak{p}}\partial_{\tau}+J\right)  ^{2}-I+(\alpha+1)%
\begin{pmatrix}
\cos^{2}\theta & \cos\theta\sin\theta\\
\cos\theta\sin\theta & \sin^{2}\theta
\end{pmatrix}
\right]  .
\]
Since $\partial_{\tau}e^{il\tau}=ile^{il\tau}$, the Hessian $\nabla_{\eta}%
^{2}\mathcal{A}$ in Fourier components is
\[
\nabla_{\eta}^{2}\mathcal{A}_{0}\left(  \theta,1,0\right)  x=\sum
_{l\in\mathbb{Z}\backslash\{0\}}A_{l}\hat{x}_{l}e^{il\tau},
\]
where%
\[
A_{l}=\left(  l^{2}+1\right)  ^{-1}\left(
\begin{array}
[c]{cc}%
(l/\mathfrak{p})^{2}+(\alpha+1)\cos^{2}\theta & -2i(l/\mathfrak{p})+(\alpha+1)\sin\theta\cos\theta\\
2i(l/\mathfrak{p})+(\alpha+1)\sin\theta\cos\theta & (l/\mathfrak{p})^{2}+(\alpha+1)\sin^{2}\theta
\end{array}
\right)  .
\]
The matrix $A_{l}$ has eigenvalues%
\[
\lambda_{l}^{\pm}=\left(  l^{2}+1\right)  ^{-1}\left(  \frac{1}{2}%
(\alpha+1)+(l/\mathfrak{p})^{2}\pm\frac{1}{2}\sqrt{\left(  \alpha+1\right)
^{2}+16(l/\mathfrak{p})^{2}}\right)  \text{.}%
\]
The eigenvalues satisfy $\lambda_{l}^{\pm}\neq0$ if $l/\mathfrak{p}\neq\sqrt{3-\alpha}%
,0$. This condition holds because $\alpha\neq2$ and $l\neq0$ for the
coordinate $\eta$. Since $\lim_{l\rightarrow\infty}\lambda_{l}^{\pm}%
=(1/\mathfrak{p})^{2}$, there is a small positive constant $C_{0}$ such that $\left\vert
\lambda_{l}^{\pm}\right\vert \geq C_{0}^{-1}$ for all $l\neq0$, i.e.
$\left\vert A_{l}^{-1}\right\vert \leq C_{0}$ for $l\neq0$. Therefore, the
estimate for $\nabla_{\eta}^{2}\mathcal{A}_{0}\left(  \theta,1,0\right)  $
follows from the fact that%
\[
\left\Vert \left(  \nabla_{\eta}^{2}\mathcal{A}_{0}(\theta,1,0)\right)
^{-1}x\right\Vert _{H_{2\pi}^{1}}^{2}=\sum_{l\in\mathbb{Z}\backslash
\{0\}}\left(  l^{2}+1\right)  \left\vert A_{l}^{-1}\hat{x}_{l}\right\vert
^{2}\leq C_{0}^{2}\sum_{l\in\mathbb{Z}\backslash\{0\}}\left(  l^{2}+1\right)
\left\vert \hat{x}_{l}\right\vert ^{2}=C_{0}^{2}\left\Vert x\right\Vert
_{H_{2\pi}^{1}}^{2}.
\]
Thus, the inverse of $\nabla_{(r,\eta)}^{2}\mathcal{A}_{0}(\theta,1,0)$ is
bounded by
\begin{equation}
C=\max\left\{  C_{0},\frac{1}{\alpha+1}\right\}  \text{.}%
\end{equation}

\end{proof}

\begin{remark}
In the gravitational case $\alpha=2$, the condition in the proposition does
not hold because $A_{l}$ is not invertible for $l=\mathfrak{p}$. This fact is a
consequence of the existence of the elliptic orbits of the Kepler problem,
i.e. the circular orbits are not isolated.
\end{remark}

\subsection{Lyapunov-Schmidt reduction}

In this section we make a Lyapunov-Schmidt reduction to finite dimension for
the operator
\[
\nabla\mathcal{A}(x;\varepsilon):\Omega_{\rho}\times\Lambda_{\varepsilon
}\subset H_{2\pi}^{1}\times\mathbb{R}\rightarrow H_{2\pi}^{1},
\]
given by
\[
\nabla\mathcal{A}(x;\varepsilon)=\nabla\mathcal{A}_{0}(x;\varepsilon
)+\nabla\mathcal{H}(x;\varepsilon),\qquad\nabla\mathcal{H}(x;\varepsilon
)=\mathcal{O}_{H_{2\pi}^{1}}(\varepsilon^{2}).
\]
Since%
\[
\nabla\mathcal{A}(\theta,r,\eta;\varepsilon)=\left(  \nabla_{\theta
}\mathcal{A},\nabla_{(r,\eta)}\mathcal{A}\right)  ,
\]
the equation $\nabla\mathcal{A}=0$ is equivalent to solve $\nabla_{\theta
}\mathcal{A}=0$ and $\nabla_{(r,\eta)}\mathcal{A}=0$. The Lyapunov-Schmidt
reduction consist of solving from the equation $\nabla_{(r,\eta)}%
\mathcal{A}=0$ defined in the set $\Omega_{\rho}\times\Lambda_{\varepsilon
_{0}}$, the components $(r,\eta)$ as function of $(\theta;\varepsilon)$,
where
\[
\Lambda_{\varepsilon_{0}}=\{\varepsilon\in\mathbb{R}:0\leq\varepsilon
<\varepsilon_{0}\}.
\]

\begin{proposition}
\label{LS}There is a $\rho$ and $\varepsilon_{0}$ such that $\mathcal{A}%
(\theta,r,\eta;\varepsilon)$ has a critical point in $\Omega_{\rho}%
\times\Lambda_{\varepsilon_{0}}$ if and only if
\[
\Psi(\theta;\varepsilon)=\mathcal{A}(\theta,r(\theta;\varepsilon),\eta
(\theta;\varepsilon);\varepsilon):S^{1}\times\Lambda_{\varepsilon_{0}%
}\rightarrow\mathbb{R}%
\]
has a critical point in $S^{1}\times\Lambda_{\varepsilon_{0}}$, where
$r=r(\theta;\varepsilon)\ $and $\eta=\eta(\theta;\varepsilon)$ is the unique
solution of
\[
\nabla_{(r,\eta)}\mathcal{A}(\theta,r,\eta;\varepsilon)=0.
\]

\end{proposition}

\begin{proof}
We have that $\nabla\mathcal{A}(\vartheta,1,0;0)=0$ for any $\vartheta\in
S^{1}$. By Proposition \ref{estimates}, the operator $\nabla_{(r,\eta)}%
^{2}\mathcal{A}(\vartheta,1,0;0)$ is invertible. Then, by the Implicit
Function Theorem, we can solve $r_{\vartheta}=r_{\vartheta}(\theta
;\varepsilon)$ and $\eta_{\vartheta}=\eta_{\vartheta}(\theta;\varepsilon)$ as
the unique solution of $\nabla_{(r,\eta)}\mathcal{A}(\theta,r,\eta
;\varepsilon)=0$ in $\mathcal{V}\times\Lambda_{\varepsilon_{\vartheta}}$,
where $\mathcal{V}$ is a neighborhood of $\vartheta$. Note that we have a
local solution around each point of $\vartheta\in S^{1}$. Using the
compactness of $S^{1}$, and the uniqueness of the solutions $r_{\vartheta
}(\theta;\varepsilon)$ and $\eta_{\vartheta}(\theta;\varepsilon)$ for
$\theta\in\mathcal{V}$ and $\varepsilon<\varepsilon_{\vartheta}$, we can
construct functions
\[
r(\theta;\varepsilon):S^{1}\times\Lambda_{\varepsilon}\rightarrow
\mathbb{R}^{+},\qquad\eta(\theta;\varepsilon):S^{1}\times\Lambda_{\varepsilon
}\rightarrow X,
\]
that solve $\nabla_{(r,\eta)}\mathcal{A}(\theta,r,\eta;\varepsilon)=0$ for
$\varepsilon<\varepsilon_{0}=\min_{\vartheta\in S^{1}}\varepsilon_{\vartheta}%
$. Therefore, critical solutions of $\mathcal{A}$ are critical solutions of
$\Psi$.
\end{proof}

\section{Comet and moon solutions}

In Proposition \ref{estimates} we assume that $\alpha\neq2$, so we analyze two
cases separately.

\subsection{Main theorems for $\alpha\geq1$ ($\alpha\neq2$)}

\begin{theorem}
\label{main}Assume that $\alpha\geq1$ ($\alpha\neq2$) and $\nabla
\mathcal{H=O}_{H_{2\pi}^{1}}(\varepsilon^{2})$, then there is a $\varepsilon
_{0}$ such that for $\varepsilon<\varepsilon_{0}$, the functional
$\mathcal{A}(x;\varepsilon)=\mathcal{A}_{0}+\mathcal{H}$ has at least two
critical solutions of the form
\[
x(\tau)=e^{J\theta_{j}}x_{0}+\mathcal{O}_{H_{2\pi}^{1}}(\varepsilon^{2})~,
\]
where $x_{0}=(1,0)$ and $\mathcal{O}_{H_{2\pi}^{1}}(\varepsilon^{2})$ is a
$2\pi$-periodic function of order $\varepsilon^{2}$.
\end{theorem}

\begin{proof}
Since $\nabla\mathcal{H=O}_{H_{2\pi}^{1}}(\varepsilon^{2})$, by standard
estimates in Lyapunov-Schmidt reduction (for instance \cite{FoGa}) we have
that $r(\theta;\varepsilon)=\mathcal{O}(\varepsilon^{2})$ and $\eta
(\theta;\varepsilon)=\mathcal{O}_{H_{2\pi}^{1}}(\varepsilon^{2})$ in
Proposition \ref{LS}. By the compactness of $S^{1}$, we conclude that $\Psi$
has at least $2$ critical points, one maximum $\theta_{1}$ and one minimum
$\theta_{2}$. Since $r(\theta;\varepsilon)=\mathcal{O}(\varepsilon^{2})$ and
$\eta(\theta;\varepsilon)=\mathcal{O}_{H_{2\pi}^{1}}(\varepsilon^{2})$, then
$\mathcal{A}(x)$ has at least two critical solutions of the form
$x(\tau)=e^{J\theta_{j}}x_{0}+\mathcal{O}_{H_{2\pi}^{1}}(\varepsilon^{2})$.
\end{proof}

In the case $\alpha\neq2$, we obtain the existence of comet solutions as an
immediate consequence of Theorem \ref{main}.

\begin{theorem}
\label{comet}Assume that $\alpha\geq1$ ($\alpha\neq2$), $\sum_{j=1}^{n}%
m_{j}=1$, $\omega=\mathfrak{p}/\mathfrak{q}$ and $\nu=1/\mathfrak{q}$. For each integer $\mathfrak{p}$ there is an integer
$\mathfrak{q}_{0}$ such that for each integer $\mathfrak{q}>\mathfrak{q}_{0}$, the restricted $(n+1)$-body
problem has at least two $2\pi \mathfrak{q}$-periodic solutions of the form
\[
q(t)=\varepsilon^{-1}e^{J\left(  \theta_{j}+\mathfrak{p}t/\mathfrak{q}\right)  }x_{0}+\mathcal{O}%
(\varepsilon),\qquad\varepsilon=\left(  \mathfrak{p}/\mathfrak{q}\right)  ^{2/(\alpha+1)}.
\]

\end{theorem}

\begin{proof}
For $\mathfrak{p}$ fixed, there is a $\mathfrak{q}_{0}\in\mathbb{Z}$ such that $\varepsilon=\left(
\mathfrak{p}/\mathfrak{q}\right)  ^{2/(\alpha+1)}<\varepsilon_{0}$ if $\mathfrak{q}>\mathfrak{q}_{0}$. From the changes of
variables \eqref{cambio_q_c} and Theorem \ref{main}, after the rescaling
$t=\tau/\nu$, for $\varepsilon<\varepsilon_{0}$ the restricted $(n+1)$-body
problem has at least two comet solutions of the form
\[
q(t)=\varepsilon^{-1}e^{J\mathfrak{p}t/\mathfrak{q}}x(\tau)=\varepsilon^{-1}e^{J\mathfrak{p}t/\mathfrak{q}}\left(
e^{J\theta_{j}}x_{0}+\mathcal{O}(\varepsilon^{2})\right)  ,
\]
where the functions $e^{J\mathfrak{p}t/\mathfrak{q}}$ and $\mathcal{O}(\varepsilon^{2})$ are $2\pi
\mathfrak{q}$-periodic.
\end{proof}

The comet solution $q(t)$ is $2\pi \mathfrak{q}$-periodic, then in one period, the $n$
primary bodies travel their $2\pi$-periodic orbit $\mathfrak{q}$ times, while the comet
winds around the origin $\mathfrak{p}$ times.

In the case of moon solutions we only need to assume that $\omega=\mathfrak{p}\nu+1$ and
$\omega^{2}=\varepsilon^{-\left(  \alpha+1\right)  /2}$, i.e. the frequency
$\omega$ can take a continuum of values and the solutions that we obtain can
be quasiperiodic,

\begin{theorem}
Assume that $\alpha\geq1$ ($\alpha\neq2$), $q_{j}(t)=e^{Jt}a_{j}$ with
\eqref{centralconf}, and $\nu=\left(  \omega-1\right)  /\mathfrak{p}$. From the changes
of variables \eqref{cambio_q_m} and Theorem \ref{main}, for $\varepsilon
<\varepsilon_{0}$ we obtain the existence of at least two moon solutions
$q(t)\ $of the form
\[
q(t)=q_{1}(t)+\varepsilon e^{J\left(  \theta_{j}+\omega t\right)  }%
x_{0}+\mathcal{O}(\varepsilon^{3}),\qquad\varepsilon=\omega^{-2/\left(
\alpha+1\right)  },
\]
where $\mathcal{O}(\varepsilon^{3})$ is a (periodic) quasiperiodic function of
order $\varepsilon^{3}$.
\end{theorem}

Actually, the moon solutions are periodic if $\omega=$ $\mathfrak{r}/\mathfrak{q}$ is
rational. For the sake of simplicity we consider only the case that $\mathfrak{p}=1$. In
those cases we have the following,

\begin{theorem}
\label{moon}Assume that $\alpha\geq1$ ($\alpha\neq2$), $q_{j}(t)=e^{Jt}a_{j}$
with \eqref{centralconf}, $\omega=$ $\mathfrak{r}/\mathfrak{q}$ and $\nu=\mathfrak{r}/\mathfrak{q}-1$. For
each integer $\mathfrak{q}$ there is an integer $\mathfrak{r}_{0}$ such that for each
integer $\mathfrak{r}>\mathfrak{r}_{0}$, the restricted $(n+1)$-body problem has at
least two $2\pi \mathfrak{q}$-periodic solutions of the form
\[
q(t)=q_{1}(t)+\varepsilon e^{J(\theta_{j}+\mathfrak{r}t/\mathfrak{q})}x_{0}+\mathcal{O}%
(\varepsilon^{3}),\qquad\varepsilon=\left(  \mathfrak{r}/\mathfrak{q}\right)  ^{-2/(\alpha
+1)}~.
\]

\end{theorem}

\begin{proof}
For $\mathfrak{q}$ fixed, there is a $\mathfrak{r}_{0}\in\mathbb{Z}$ such that
$\varepsilon=\left(  \mathfrak{r}/\mathfrak{q}\right)  ^{-2/(\alpha+1)}<\varepsilon_{0}$ if
$\mathfrak{r}>\mathfrak{r}_{0}$. Thus after the rescaling $t=\tau/\nu$, we obtain solutions
\[
q(t)=q_{1}(t)+\varepsilon e^{J\mathfrak{r}t/\mathfrak{q}}\left(  e^{J\theta_{j}}x_{0}+\mathcal{O}%
(\varepsilon^{2})\right)  ,
\]
where the function $\mathcal{O}(\varepsilon^{2})$ is $2\pi/\nu$-periodic in
$t$. Since $e^{J\omega t}$ and $\mathcal{O}(\varepsilon^{2})$ are $2\pi
\mathfrak{q}$-periodic, then the product $e^{J\omega t}\mathcal{O}(\varepsilon^{2})$ is
$2\pi \mathfrak{q}$-periodic.
\end{proof}

In the period $2\pi \mathfrak{r}$ the moon follows a small-amplitude circular orbit of
the Kepler problem around the first primary body, where the moon winds around
this primary body $\mathfrak{r}$ times, while the primary bodies travel their $2\pi
$-periodic orbit $\mathfrak{q}$ times.

\subsection{Main theorems for $\alpha=2$}

In the case $\alpha=2$, the matrices $A_{\mathfrak{p}}$ and $A_{-\mathfrak{p}}$ in Proposition
\ref{estimates} are not invertible. To avoid this problem we define the action
of the group $\mathbb{Z}_{m\mathfrak{p}}$ generated by $\zeta$,
\[
\zeta x(t)=x\left(  t-2\pi/m\mathfrak{p}\right)  ,\qquad m\geq 2.
\]
The action functional $\mathcal{A}_{0}$ is $\mathbb{Z}_{m\mathfrak{p}}$-invariant because
it does not depend on time explicitly. Therefore, the functional
$\mathcal{A=A}_{0}+\mathcal{H}$ is $\mathbb{Z}_{m\mathfrak{p}}$-invariant when
$\mathcal{H}$ is $\mathbb{Z}_{m\mathfrak{p}}$-invariant.

\begin{lemma}
The action functional $\mathcal{H}$ for the comet is $\mathbb{Z}_{m\mathfrak{p}}%
$-invariant if there is permutation $\sigma\in S_{n}$ of the set $\{1,...,n\}$
such that
\begin{equation}
m_{j}=m_{\sigma(j)},\qquad q_{j}(t+2\pi \mathfrak{q}/m\mathfrak{p})=e^{2\pi J/m}q_{\sigma
(j)}(t),\qquad j=1,...,n\text{.} \label{1}%
\end{equation}
In the case of moon, the action $\mathcal{H}$ is $\mathbb{Z}_{m\mathfrak{p}}$-invariant
if the permutation $\sigma$ satisfies $\sigma(1)=1$, $a_{1}=0$ and
\begin{equation}
m_{j}=m_{\sigma(j)},\qquad a_{j}=e^{2\pi J /m}a_{\sigma(j)},\qquad
j=2,...,n\text{.} \label{2}%
\end{equation}

\end{lemma}

\begin{proof}
In the comet and moon problems, the action functional $\mathcal{H}(x)=\int
_{0}^{2\pi}h(x,\tau)d\tau$ depends on $\tau$ only through the functions
$x_{j}(\tau)$. In the comet case%
\[
\mathcal{H}(\zeta x)=\int_{0}^{2\pi}\sum_{j=1}^{n}m_{j}\left[  \phi_{\alpha
}(\left\Vert x(\tau)-\varepsilon x_{j}(\tau+2\pi/m)\right\Vert )-\phi_{\alpha
}(\left\Vert x(\tau)\right\Vert )\right]  d\tau=\mathcal{H}(x),
\]
if $m_{j}=m_{\sigma(j)}$ and $x_{j}(\tau+2\pi/m\mathfrak{p})=x_{\sigma(j)}(\tau)$ for a
permutation $\sigma$ of the elements $j\in\{1,...,n\}$. Since $x_{j}%
(\tau)=e^{-J\mathfrak{p}\tau}q_{j}(\mathfrak{q}\tau)$, because $\omega=\mathfrak{p}/\mathfrak{q}$ 
and $\nu=1/\mathfrak{q}$, this holds when
\[
e^{-2\pi J/m}e^{-J\mathfrak{p}\tau}q_{j}(\mathfrak{q}\tau+2\pi \mathfrak{q}/m\mathfrak{p})=x_{j}(\tau+2\pi/m)=x_{\sigma
(j)}(\tau)=e^{-J\mathfrak{p}\tau}q_{\sigma(j)}(\mathfrak{q}\tau).
\]
This condition is equivalent to $q_{j}(\mathfrak{q}\tau+2\pi \mathfrak{q}/m\mathfrak{p})=e^{2\pi J/m}%
q_{\sigma(j)}(\mathfrak{q}\tau)$ and to condition \eqref{1}.

The same result holds in the moon case if there is$\ $a permutation $\sigma$
such that $\sigma(1)=1$ and $x_{j}(\tau+2\pi/m\mathfrak{p})=x_{\sigma(j)}(\tau)$. Since
$(\omega-1)/\nu=\mathfrak{p}$, in this case we have $x_{j}(\tau)=e^{-J\mathfrak{p}\tau}a_{j}$.
Therefore, we require that%
\[
e^{-2\pi J/m}e^{-J\mathfrak{p}\tau}a_{j}=x_{j}(\tau+2\pi/m\mathfrak{p})=x_{\sigma(j)}(\tau
)=e^{-J\mathfrak{p}\tau}a_{\sigma(j)}.
\]
This condition holds only when $a_{j}=e^{2\pi J/m}a_{\sigma(j)}$. Since
$\sigma(1)=1$, we require that $a_{1}=0$.
\end{proof}

\begin{theorem}
\label{grav}Under the assumptions \eqref{1} for the comet solutions, and
\eqref{2} for the moon solutions, the same results as in Theorems \ref{comet}
and \ref{moon} hold in the gravitational case with $\alpha=2$.
\end{theorem}

\begin{proof}
The subspace of fixed points under the $\mathbb{Z}_{m\mathfrak{p}}$-action is
\[
\left(  H_{2\pi}^{1}\right)  ^{\mathbb{Z}_{m\mathfrak{p}}}=\left\{  x:\sum_{l\in
\mathbb{Z}}x_{j}e^{il(t+\pi)}=\sum_{l\in\mathbb{Z}}x_{j}e^{ilt}\right\}
=\left\{  x:x_{l}=0\quad\text{for}\quad l=0,\pm m\mathfrak{p},\pm2m\mathfrak{p},...\right\}
\]
By the Palais Criticality Principle, a critical point of $\mathcal{A}:\left(
H_{2\pi}^{1}\right)  ^{\mathbb{Z}_{m\mathfrak{p}}}\rightarrow\mathbb{R}$ is a critical
point of the action $\mathcal{A}:H_{2\pi}^{1}\rightarrow\mathbb{R}$. Then, we
can apply the previous method in the fixed point space $\left(  H_{2\pi}%
^{1}\right)  ^{\mathbb{Z}_{m\mathfrak{p}}}$. The result follows from the fact that the
Hessian $\nabla^{2}\mathcal{A}_{0}(e^{J\theta}a_{0})$ is invertible in
$\left(  H_{2\pi}^{1}\right)  ^{\mathbb{Z}_{m\mathfrak{p}}}$ if $m\geq2$.
\end{proof}

\begin{remark}
\label{ReC}For example, the choreographies in \cite{CaDoGa18b} have $n$ bodies
forming $m$-polygons at any time. In these solutions there is a permutation
$\sigma$ with $\sigma^{m}=(1)$ such that
\begin{equation}
q_{j}(t)=e^{2\pi J/m}q_{\sigma(j)}(t). \label{Ch}%
\end{equation}
The condition \eqref{1} is satisfied by the choreographies with the condition
\eqref{Ch} for $\mathfrak{q}=lm\mathfrak{p}$ with $l\in\mathbb{N}$. A particular example of such
choreographies is the super-eight choreography (see Figure 2) that satisfies
the symmetry
\[
q_{j}(t)=e^{J\pi}q_{\sigma(j)}(t)=-q_{\sigma(j)}(t),
\]
where $\sigma=(13)(24)$ is the permutation of order $m=2$.
\end{remark}

\begin{remark}
\label{ReM}An example of a central configuration that satisfies condition
\eqref{2} is the Maxwell configuration, which consists of a central mass
$a_{1}=0$ and $(n-1)$-equal masses at the vertices of a polygon $a_{j}%
=e^{2\pi J/m}x_{0}$ with $m=n-1$. The Maxwell configuration satisfies
condition \eqref{2} for the permutation $\sigma=(2~3~\ldots~n-1)$ of order
$m$. Similar configurations that satisfy the condition \eqref{2} consist of a
central mass with an arrangement of nested polygons in \cite{GaIz11} .
\end{remark}

\section{Numerical study of periodic orbits}

The variational method in the previous section has some limitations due to
assumptions imposed on the solution of the primary bodies. This section is
dedicated to exploring numerically, with the reversibility technics in
\cite{Munoz2006}, the existence of comet and moon solutions for four primaries
following the super-eight choreography. The super-eight choreography is a
periodic solution of the 4-body problem where four primaries with unitary mass
follow periodically the same path. We assume that the origin of the coordinate
system is located at the center of mass.

\subsection{Super-eight choreography of the 4-body problem}

The super-eight choreography has two special kinds of configurations that
correspond to fixed points of transformations for positions and velocities.
These special transformations are called \emph{reversing symmetries}
\cite{Lamb1998}, and we refer to its fixed points as \emph{reversible
configurations}. We use this technique to obtain periodic orbits in the
restricted 5-body problem.

During the temporal evolution of the super-eight choreography two kind of
reversible configurations appear at $\overline{T}:=\pi/4$-units of time, where
$T:=2\pi$ is the period of the orbit. One of these reversible configurations,
called \emph{isosceles configuration}, has the characteristic that the bodies
are located at the vertices of two isosceles triangles where one vertex is the
center of mass of the system. The velocities of the bodies are related by a
reflection along the basis of the corresponding isosceles triangle. The
positions $q_{j}$ and velocities $v_{j} = \dot{q}_{j}$ of the isosceles
configurations, for time $t=0$, is given by
\begin{equation}%
\begin{array}
[c]{c}%
{q}_{1}=(0.939977120285667,-0.327721385645527)^{T},\\[6pt]%
{q}_{2}=Kq_{1},\quad{q}_{3}=-{q}_{1},\quad{q}_{4}=-{q}_{2},\\[6pt]%
v_{1}=(1.122200245052303,-0.117392625737923)^{T},\\[6pt]%
v_{2}=-K v_{1},\quad v_{3}=- v_{1},\quad v_{4}=- v_{2},
\end{array}
\label{cieightiso}%
\end{equation}
where%
\begin{equation}
K=\left(
\begin{array}
[c]{cc}%
1 & 0\\[6pt]%
0 & -1
\end{array}
\right)  . \label{matrix}%
\end{equation}

At time $t=\overline{T}$ the four bodies pass by other kind of reversible
configuration. The geometric property of this configuration is that the
positions and velocities of the bodies are orthogonal; we called it
\emph{orthogonal configuration}. The positions and velocities of the four
bodies at $t=\overline{T}$ are
\begin{equation}%
\begin{array}
[c]{c}%
{q}_{1}=(1.382856843618412,0)^{T},\quad{q}_{2}=(0,0.157029922281204)^{T}%
,\\[6pt]%
{q}_{3}=-{q}_{1},\quad{q}_{4}=-{q}_{2},\\[6pt]%
v_{1}=(0,0.584872630814899)^{T}, v_{2}=(1.871935245878693,0)^{T},\\[6pt]%
v_{4}=- v_{2},\quad v_{3}=- v_{1}.
\end{array}
\label{cieightort}%
\end{equation}

The isosceles and orthogonal reversible configurations of the Gerver's super
eight are shown in Fig. \ref{figure2}. We remark that these configurations are
fixed points of reversing symmetries in the 4-body problem with equal masses
\cite{Munoz2006}.

\begin{figure}[h]
\centering
\captionsetup{width=.9\linewidth} \includegraphics[width=75mm]{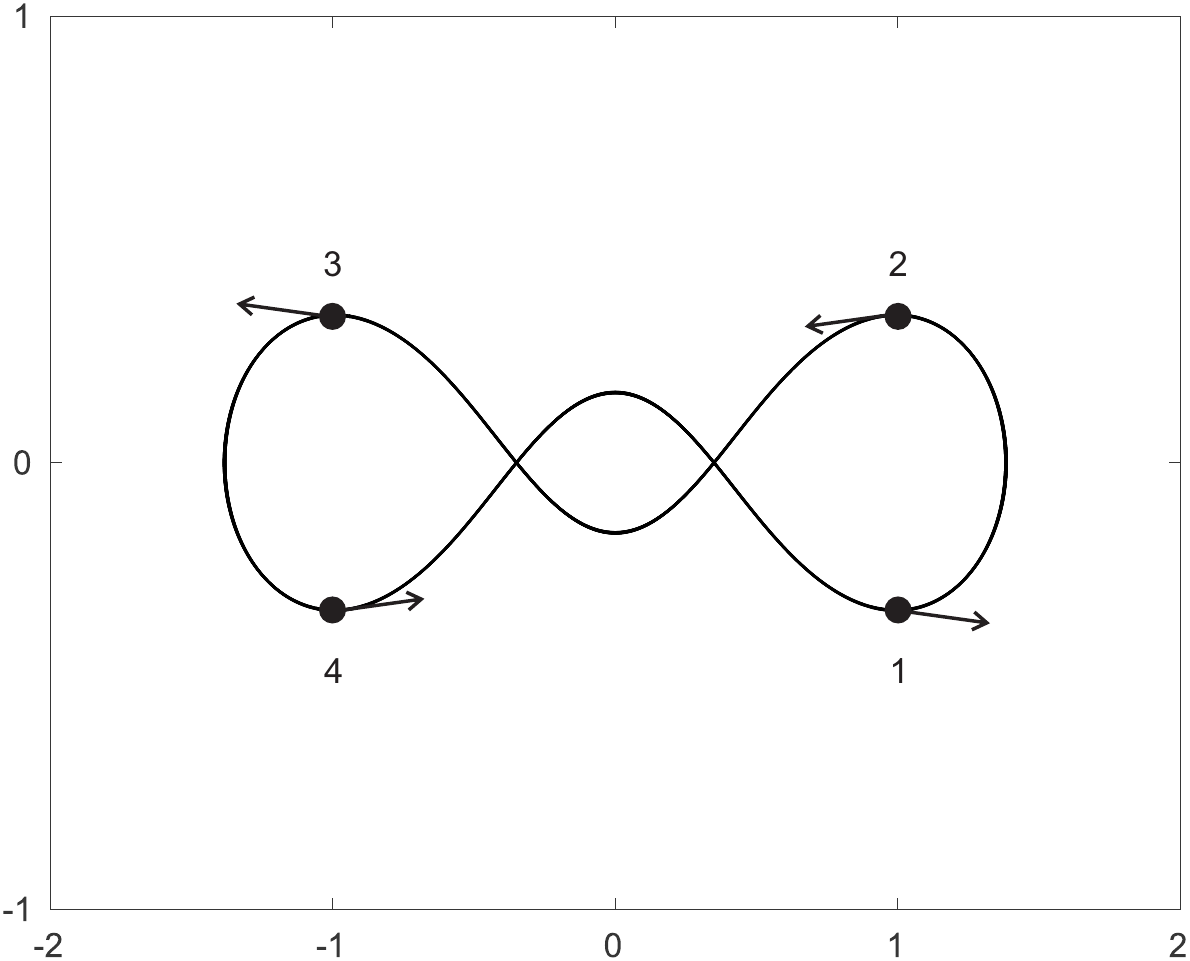}
\includegraphics[width=75mm]{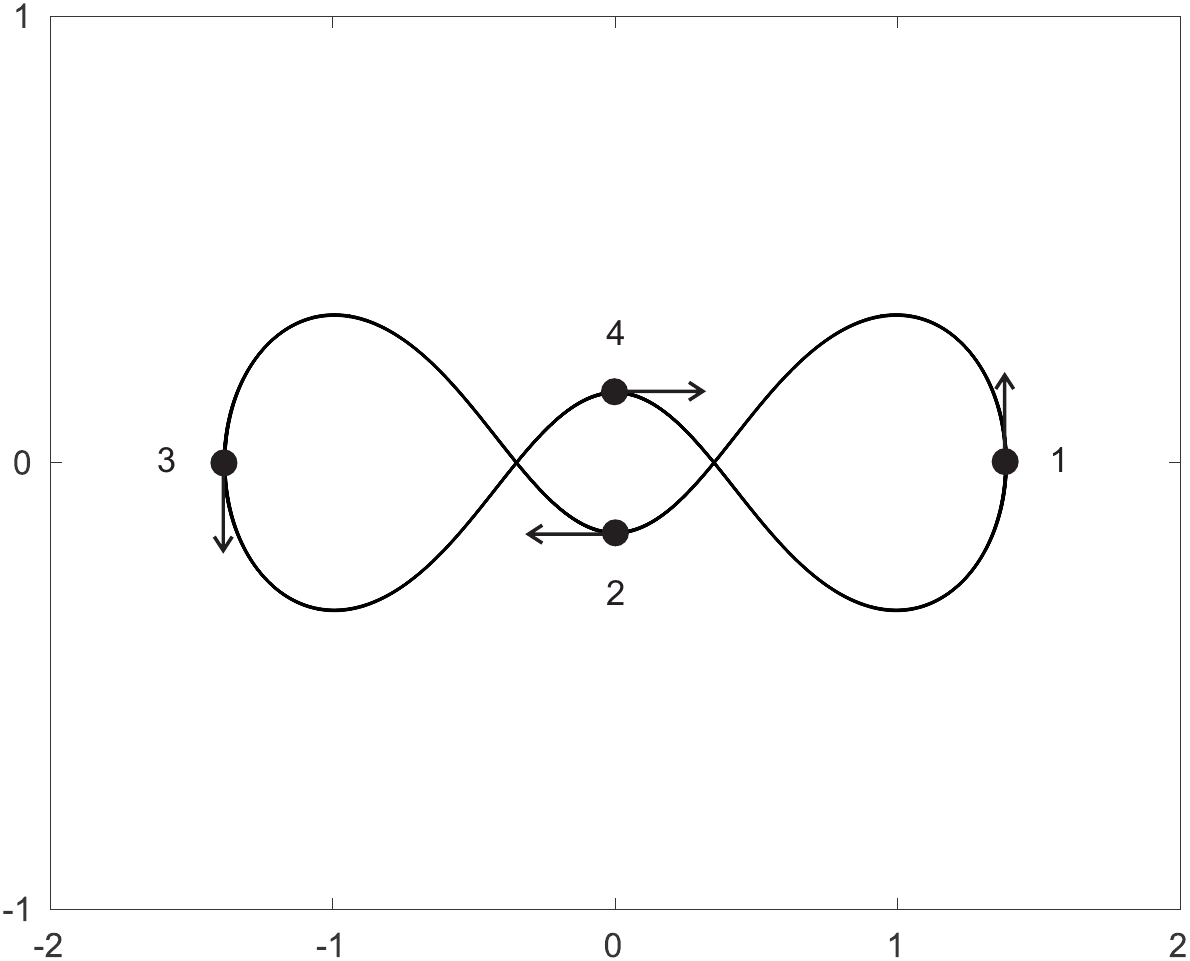} \caption{Gerver's super eight
choreography. At left an isosceles reversible configuration, at time $t=0$. At
right an orthogonal reversible configuration, at time $t = \overline{T}$.}%
\label{figure2}%
\end{figure}

\subsection{The restricted 5-body problem for the super-eight choreography}

Consider a system of five particles, four of them with unitary mass $m_{i}=1$,
$i=1,2,3,4$, following the periodic super-eight choreography, and the fifth
$i=5$ as a test particle moving under the force exerted by the other four. The
equations of motions of the five particles are given by the expressions
\begin{equation}
\ddot{q}_{i}(t)=-\sum_{j=1,j\neq i}^{5}m_{j}\frac{q_{i}(t)-q_{j}%
(t)}{\left\Vert q_{i}(t)-q_{j}(t)\right\Vert ^{3}},\enskip i=1,\cdots,5.
\label{eqmot}%
\end{equation}
We assume that the origin of the coordinate system coincides with the center
of mass of the choreographic bodies. We restrict \eqref{eqmot} in such a way
the bodies with unitary mass always follow the super-eight 
choreography. So, for the bodies with unitary mass we choose as initial
conditions \eqref{cieightiso} or \eqref{cieightort}. In consequence, only the
initial conditions of the fifth particle are unknown. Our aim is to determine
initial conditions of the fifth particle that lead to periodic orbits of comet
or moon type. We will consider reversible configurations of the restricted
5-body problem that are consistent with the Gerver's super-eight choreography
as we explain in the following sections.

\subsection{Reversing symmetries of the restricted five-body problem}

In order to introduce the results related with reversing symmetries, we need
to rewrite the equations of motion \eqref{eqmot}. For this aim, we introduce
the state vector of the five bodies
\begin{equation}
u(t)=(q_{1}(t),\cdots,q_{5}(t), v_{1}(t),\cdots, v_{5}(t)). \label{u}%
\end{equation}
The equations of motion can be written as
\[
\dot{u}(t)=( v_{1}(t),\cdots, v_{5}(t),a_{1}(t),\cdots,a_{5}(t)),
\]
where
\begin{equation}
{a}_{i}(t)=-\sum_{j=1,j\neq i}^{5}m_{j}\frac{q_{i}(t)-q_{j}(t)}{\left\Vert
q_{i}(t)-q_{j}(t)\right\Vert ^{3}},\quad i=1,\cdots,5. \label{acel}%
\end{equation}
With this, the equation \eqref{eqmot} becomes
\begin{equation}
\frac{d{u}(t)}{dt}=F(u(t)),\quad(F_{2i-1},F_{2i})=q_{i}(t),\quad
(F_{8+2i-1},F_{8+2i})=a_{i}(t),\quad i=1,\cdots,5. \label{odeF}%
\end{equation}
We notice that system \eqref{odeF} is not defined at collisions. That is,
defining
\[
\Delta_{ij}=\{(q_{1},\cdots,q_{5})\in\mathbb{R}^{10}\quad|\quad q_{i}%
=q_{j}\},\qquad\Delta=\bigcup_{i<j}\Delta_{ij},
\]
we have that the system \eqref{odeF} is defined only for $u\in\Omega
=\mathbb{R}^{10}\setminus\Delta\times\mathbb{R}^{10}$.

\begin{definition}
Given the ordinary differential equation \eqref{odeF}, we say that an
involution $R:\Omega\to\Omega$ is a reversing symmetry if
\[
\frac{d}{dt}R({u}(t)) = -F(R(u(t)))
\]
holds. The set of fixed points of $R$ is denoted as $\mathbf{Fix}(R)$.
\end{definition}

The reversing symmetries are useful for determining symmetric periodic orbits.
The Mu\~{n}oz-Almaraz's Theorem \cite{Munoz1998} gives the relation between
the reversing symmetries and symmetric periodic orbits. We state this theorem
applied to our restricted 5-body problem.

\begin{theorem}
\label{rev} Let $u(t)=(q_{1}(t),\cdots,q_{5}(t), v_{1}(t),\cdots,{v}_{5}(t))$
be a well-defined solution of the 5-body problem in the time interval
$[0,T_{0}]$. Suppose that the solution $u(t)$ passes through fixed points of
reversing symmetries $R$ and $\widehat{R}$ at times $t=0$ and $t=T_{0}$,
respectively. Then the solution is defined for all $t\in\mathbb{R}$, and we
have for all $m\in\mathbb{Z}$ that
\[%
\begin{array}
[c]{c}%
u(-t)=Ru(t),\\[6pt]%
u(t)=\widehat{R}u(2T_{0}-t),\\[6pt]%
u(2mT_{0}+t)=(\widehat{R}R)^{m}u(t).
\end{array}
\]
In addition, if there exists some $M\in\mathbb{N}$ such that $(\widehat
{R}R)^{M}=id$ then $u(t)$ is periodic, with period $T=2MT_{0}$.
\end{theorem}

In order to apply Theorem \ref{rev} to the restricted 5-body problem, we need
to consider fixed points of reversing symmetries of that problem. The
super-eight choreography is a solution of the 4-body problem that passes
through fixed points of specific reversing symmetries of the 4-body problem.
In our case, we consider reversing symmetries of the 5-body problem that are
consistent with the super-eight choreography. We remark that these reversing
symmetries of the 5-body problem also hold for our restricted 5-body problem.
There are several reversing symmetries that allow to obtain periodic orbits of
comet and moon type. We only consider the reversing symmetries
\[
\Phi_{1x}:%
\begin{array}
[c]{ll}%
q_{1}\rightarrow Kq_{2}, & v_{1}\rightarrow-Kv_{2},\\
q_{2}\rightarrow Kq_{1}, & v_{2}\rightarrow-Kv_{1},\\
q_{3}\rightarrow Kq_{4}, & v_{3}\rightarrow-Kv_{4},\\
q_{4}\rightarrow Kq_{3}, & v_{4}\rightarrow-Kv_{3},\\
q_{5}\rightarrow Kq_{5}, & v_{5}\rightarrow-Kv_{5},
\end{array}
\quad\Phi_{1y}:%
\begin{array}
[c]{ll}%
q_{1}\rightarrow Kq_{2}, & v_{1}\rightarrow-Kv_{2},\\
q_{2}\rightarrow Kq_{1}, & v_{2}\rightarrow-Kv_{1},\\
q_{3}\rightarrow Kq_{4}, & v_{3}\rightarrow-Kv_{4},\\
q_{4}\rightarrow Kq_{3}, & v_{4}\rightarrow-Kv_{3},\\
q_{5}\rightarrow-Kq_{5}, & v_{5}\rightarrow Kv_{5},
\end{array}
\]%
\[
\Psi_{1x}:%
\begin{array}
[c]{ll}%
q_{1}\rightarrow Kq_{1}, & v_{1}\rightarrow-Kv_{1},\\
q_{2}\rightarrow-Kq_{2}, & v_{2}\rightarrow Kv_{2},\\
q_{3}\rightarrow Kq_{3}, & v_{3}\rightarrow-Kv_{3},\\
q_{4}\rightarrow-Kq_{4}, & v_{4}\rightarrow Kv_{4},\\
q_{5}\rightarrow Kq_{5}, & v_{5}\rightarrow-Kv_{5},
\end{array}
\quad\Psi_{1y}:%
\begin{array}
[c]{ll}%
q_{1}\rightarrow Kq_{1}, & v_{1}\rightarrow-Kv_{1},\\
q_{2}\rightarrow-Kq_{2}, & v_{2}\rightarrow Kv_{2},\\
q_{3}\rightarrow Kq_{3}, & v_{3}\rightarrow-Kv_{3},\\
q_{4}\rightarrow-Kq_{4}, & v_{4}\rightarrow Kv_{4},\\
q_{5}\rightarrow-Kq_{5}, & v_{5}\rightarrow Kv_{5},
\end{array}
\]
where $K$ is the matrix defined in \eqref{matrix}. The set Fix$(\Phi_{1x})$ is
given by the points \eqref{u}, where $q_{j}=(q_{jx},q_{jy})$, $v_{j}=( v_{jx},
v_{jy})$, $j=1,\cdots,5$ satisfy
\begin{equation}%
\begin{array}
[c]{l}%
q_{2x}=q_{1x},\enskip q_{2y}=-q_{1y},\enskip q_{3x}=-q_{1x},\enskip q_{3y}%
=-q_{1y},\enskip q_{4x}=-q_{2x},\enskip q_{4y}=-q_{2y},\enskip q_{5y}=0,\\
v_{2x}=- v_{1x},\enskip v_{2y}= v_{1y},\enskip\dot{q}_{3x}=- v_{1x}%
,\enskip v_{3y}=- v_{1y},\enskip v_{4x}=- v_{2x},\enskip v_{4y}=-
v_{2y},\enskip v_{5x}=0.
\end{array}
\label{revt=0}%
\end{equation}
On the other hand, for the set Fix$(\Psi_{1x})$ we have
\begin{equation}%
\begin{array}
[c]{l}%
q_{1y}=0,\enskip q_{2x}=0,\enskip q_{3x}=-q_{1x},\enskip q_{3y}=-q_{1y}%
,\enskip q_{4x}=-q_{2x},\enskip q_{4y}=-q_{2y},\enskip q_{5y}=0,\\
v_{1x}=0,\enskip v_{2y}=0,\enskip v_{3x}=- v_{1x},\enskip v_{3y}=-
v_{1y},\enskip v_{4x}=- v_{2x},\enskip v_{4y}=- v_{2y},\enskip v_{5x}=0.
\end{array}
\label{revt=1}%
\end{equation}

The sets Fix$(\Phi_{1y})$ and Fix$(\Psi_{1y})$ are obtained from
\eqref{revt=0} and \eqref{revt=1}, respectively, by replacing the conditions
for the fifth particle by ${q}_{5x}=0$, $v_{5y}=0$. We say that Fix$(\Phi
_{1x})$ and Fix$(\Phi_{1y})$ are \emph{isosceles reversible configurations},
whereas Fix$(\Psi_{1x})$ and Fix$(\Psi_{1y})$ \emph{orthogonal reversible
configurations}, of the restricted five-body problem. We remark that these
reversible configurations hold for both full and restricted 5-body problems.
For the restricted 5-body problem, we are interested in those orbits that pass
through two fixed points of some reversing symmetries. We will consider
specific combinations of reversing symmetries that lead to comet and moon orbits.

\subsection{Comet orbits}

For the numerical computation of comet orbits of the restricted 5-body
problem, we consider two different combinations of reversible configurations.
First, we deal with the solutions that pass through points within
Fix$(\Phi_{1x})$ at $t=0$, and Fix$(\Phi_{1y})$ at some time $t=T_{0}$.
According to the initial condition of the choreographic bodies
\eqref{cieightiso}, the orbit of the four primary bodies is consistent with
Fix$(\Phi_{1y})$ if and only if $T_{0}=4m\overline{T}$, $m\in\mathbb{Z}$.
Thus, if we have a solution $u(t)$ which meet $u(0)\in$ Fix$(\Phi_{1x})$,
$u(T_{0})\in$ Fix$(\Phi_{1y})$ with $T_{0}=4m\overline{T}$ for some
$m\in\mathbb{N}$, by Theorem \ref{rev} and the relation between the reversing
symmetries
\[
(\Phi_{1y}\Phi_{1x})^{2}=id,
\]
we get that $u(t)$ is periodic (in inertial frame) with period $T=4T_{0}$. For
a second combination of reversing symmetries consider $u(t)$ such that
$u(0)\in$ Fix$(\Phi_{1x})$, $u(T_{0})\in$ Fix$(\Phi_{2y})$ with $T_{0}%
=(4m+2)\overline{T}$ for some $m\in\mathbb{N}$. In this case the orbit is
periodic since $(\Phi_{2y}\Phi_{1x})^{2}=id$, with period $T=4T_{0}$.

In order to compute numerically the initial conditions of $u(t)$ for the first
combination (for the second one we could follow a similar approach), we need
to find a time $T_{0}=4m\overline{T}$ with $m\in\mathbb{N}$, and initial
conditions $q_{5} = (\alpha,0)$, $v_{5} = (0,\beta)$, $\alpha$, $\beta
\in\mathbb{R}$, in such a way the equalities%
\begin{equation}%
\begin{array}
[c]{l}%
\phi_{0}(\alpha,0,0,\beta)=(\alpha,0,0,\beta),\\
\phi_{T_{0}}(\alpha,0,0,\beta)=(\gamma,0,0,\delta),
\end{array}
\label{numiso}%
\end{equation}
hold. Here $\phi$ is the flow associated to the equations of motions
(restricted problem), and $\gamma$, $\delta$ are arbitrary real numbers. The
first equation \eqref{numiso} establish the condition $u(0)\in$ Fix$(\Phi
_{1x})$, whereas the second one $u(T_{0})\in$ Fix$(\Phi_{1y})$. In this case
both $q_{5}=(\alpha,0)$, $v_{5}=(0,\beta)$ and \eqref{cieightiso} define an
initial condition of a periodic orbit of comet type in the restricted 5-body
problem. For details about the numerical computation of these orbits the
reader is referred to the restricted $4$-body problem in \cite{Lara2019}.

The main property of the comet orbits is that the test particle is located far
away from the primaries. In that case, the primaries exert a force
approximately equivalent to that of a body with mass $m=4$ over the test
particle. In Figure \ref{figure1} we have shown several orbits of this type
and in Table \ref{table-1} the corresponding initial conditions.

\begin{table}[ptb]
\centering
\captionsetup{width=.9\linewidth}
\begin{tabular}
[c]{|c|c|c|}\hline
\rule{0pt}{4ex} $\displaystyle \overline{T}$ & $\alpha$ & $\beta$\\\hline
$2\pi$ & 4.116104103490420 & 1.044999754887220\\\hline
$5\pi/2$ & 4.742060123223827 & 0.958945634262276\\\hline
$3\pi$ & 5.330615961036938 & 0.896037359621114\\\hline
$7\pi/2$ & 5.889293694917488 & 0.847128753375993\\\hline
$4\pi$ & 6.423300718815878 & 0.807515201172657\\\hline
$\pi/2$ & 1.469992697921058 & 3.966907060848269\\\hline
\end{tabular}
\caption{Positions and velocities $q_{5}=(\alpha,0)$, $v_{5}=(0,\beta)$ of the
fifth particle that give rise to periodic orbits in the restricted 5-body
problem. The first five rows and \eqref{cieightiso} define initial conditions
of comet orbits. On the other hand, the sixth row and \eqref{cieightort}
define initial conditions of moon orbits.}%
\label{table-1}%
\end{table}

\subsection{Moon orbits}

The moon orbits are the opposite of the comet orbits, they are characterized
by the property that the test particle moves around some of the primaries,
acting as satellite. In this case these bodies conform approximately a binary
system. We assume that the primary involved with the test particle is the body
with index $i=1$.

In order to compute periodic orbits of moon type we also use the technic of
reversing symmetries. In this case the isosceles reversible configuration of
the restricted 5-body problem is not consistent with the super-eight
choreography and the moon orbits. Nevertheless, the orthogonal reversible
configurations are consistent with the super-eight choreography and the moon
orbits, which are precisely the fixed points of $\Psi_{1x}$ and $\Psi_{1y}$.
In order to define a boundary value problem where the initial condition is
involved, we make a temporal shift of $\overline{T}$ units in the equations of
motion \eqref{eqmot}. Therefore, by means of introducing the new time
$\tau=t-\overline{T}$, the orthogonal reversible configuration which happens
at $t = \overline{T}$ is associated to the new time $\tau=0$. Thus, consider a
solution $u(\tau)$ in such a way $u(0)\in$ Fix $(\Psi_{1x})$ and $u(T_{0})\in$
Fix $(\Psi_{1y})$, which can happen only if $T_{0}=4m\overline{T}$ for some
$m\in\mathbb{N}$. According to Theorem \ref{rev}, and the relation
\[
(\Psi_{1y}\Psi_{1x})^{2}=id,
\]
the solution $u(\tau)$ is periodic (in inertial frame), with period $T=4T_{0}$.

In order to compute the initial conditions of periodic orbits, we proceed as
we did for the comet case (see \eqref{numiso}). We have to find a time
$T_{0}=4m\overline{T}$ with $m\in\mathbb{N}$, and values $\alpha,\beta
\in\mathbb{R}$, in such a way that the equalities
\begin{equation}%
\begin{array}
[c]{l}%
\phi_{0}(\alpha,0,0,\beta)=(\alpha,0,0,\beta),\\
\phi_{T_{0}}(\alpha,0,0,\beta)=(0,\gamma,\delta,0),
\end{array}
\label{numort}%
\end{equation}
are fulfilled, where $\gamma$, $\delta$ are arbitrary real numbers. The first
equation \eqref{numort} implies that $u(0)\in$ Fix$(\Psi_{1x})$, whereas the
other one that $u(T_{0})\in$ Fix$(\Psi_{1y})$. The initial condition of the
restricted 5-body problem (with shifted time) is given by $q_{5}=(\alpha,0)$,
$v_{5}=(0,\beta)$ and \eqref{cieightort}. In Figure \ref{figure1} we have
shown a periodic and symmetric moon orbit of the restricted 5-body problem; in
Table \ref{table-1} we give the corresponding initial condition.


\end{document}